\documentclass[1pt]{article}
\usepackage{amsmath, amssymb, amsthm, amsfonts, cases}
\usepackage{mathrsfs}
\usepackage{url}
\usepackage{authblk}
\usepackage[usenames]{color}
\usepackage{geometry}

\newtheorem{rmk}{Remark}[section]
\newtheorem{lemma}{Lemma}[section]
\newtheorem{theorem}{Theorem}[section]

\newtheorem{definition}{Definition}[section]
\newtheorem{pro}{Problem}[section]
\newtheorem{ass}{Assumption}[section]

\renewcommand{\d}{\mathrm{d}}
\newcommand{\eps}{\varepsilon}

\newcommand{\la}{\langle}
\newcommand{\ra}{\rangle}

\allowdisplaybreaks

 \makeatother
\begin{document}

\title{  Optimal Control of Forward-Backward Stochastic Differential System of Jump Diffusion with Observation Noise: Stochastic
Maximum Principle
 \thanks{This work was supported by the Natural Science Foundation of Zhejiang Province
for Distinguished Young Scholar  (No.LR15A010001),  and the National Natural
Science Foundation of China (No.11471079, 11301177) }}

\date{}

   \author{ Qingxin Meng\thanks{Corresponding author.   E-mail: mqx@zjhu.edu.cn}
\hspace{1cm}
\\\small{Department of Mathematics, Huzhou University, Zhejiang 313000, China}}

\maketitle
\begin{abstract}

This paper is concerned with the partial information optimal control problem  of 
wa controlled forward-backward stochastic differential equation of jump diffusion with correlated noises between the system and the observation.
 For this type of partial information optimal control problem, Necessary and sufficient optimality conditions, in the form of Pontryagin maximum principle, for the partial information optimal control
are established using  a unified way.
Moreover, our admissible control process $u(\cdot)$ satisfies the following
 integrable condition condition:
 \begin{equation*} \label{eq:1.16}
  \mathbb E\bigg[\int_0^T|u(t)|^4
dt\bigg]<\infty,
\end{equation*}
\end{abstract}


\maketitle

\section{ Introduction}
In recent years,  the control 
problem of forward-backward 
stochastic system with 
observation and their applications in mathematical finance have been studied extensively, see for example, \cite{ML2016, Shi2010,SZ2013, WW2009, WWX2015,WWX,WWZ,xiao2011,XZZ2016, ZRW}
 
In  the above references, one of  the most important results was established by  Wang, Wu and Xiong \cite{WWX},  where  a partial information optimal control problem derived by forward-backward stochastic systems driven by Brownian motion with correlated noises between the system and the observation. Utilizing a direct method, an approximation method, and a Malliavin derivative method, they established three versions of maximum principle (i.e., necessary condition) for optimal control in the sense of  weak solutions,
where in order to  establishing 
the corresponding the 
variation formula for the 
cost functional,
the following  $L^8-$ bound
is imposed on the admissible controls:
\begin{equation} \label{eq:1.17}
  \sup_{t\in [0, T]}\mathbb E\bigg[|u(t)|^8
dt\bigg]<\infty.
\end{equation}
 In some reference(see, for example \cite{SZ2013,xiao2011, ZRW}),
the following  $L^2-$ bound
is imposed on the admissible controls:
\begin{equation} \label{eq:1.17}
  \sup_{t\in [0,T]}\mathbb E\bigg[|u(t)|^2
dt\bigg]<\infty.
\end{equation}
In fact, $L^2-$ bound on the admissible 
 does not seem enough to obtain the
 corresponding variation formula for the cost functional because
 the stochastic process $\rho(\cdot)$  (see
 \eqref{eq:8}) as an additional 
 state process  is  multiplied by performance indicators as follows:
 \begin{equation*}\label{eq:15}
\begin{split}
J(u(\cdot ))=& \mathbb E\displaystyle\bigg[%
\int_{0}^{T}\rho^u(t)l(t,x(t),y(t),z_1(t),z_2(t),
\Lambda(t,\cdot), u(t))dt +\rho^u(T)\Phi(x(T))
+\gamma(y(0))\bigg].
\end{split}
\end{equation*}
where $l$ and $\Phi$
are  quadratic 
growth with respect to 
the state process $x(\cdot).$
Therefore, in order to obtaining
the well-definedness of the 
cost functional and the corresponding 
variation formula, we should at least put forward the following conditions
on the admissible controls:
\begin{equation} \label{eq:1.1118}
  \mathbb E\bigg[\int_0^T|u(t)|^2
dt\bigg]^{{1+\delta}}<\infty,~for~some~ \delta >0.
\end{equation}
or
\begin{equation} \label{eq:1.18}
  \sup_{t\in [0, T]}\mathbb E\big[|u(t)|\big]^{{2+\delta}}<\infty,~for~some~ \delta >0.
\end{equation}
 In 2017, Meng, Shi and Tang \cite{MST} revisits the partial information optimal control problem  considered by Wang, Wu and
Xiong \cite{WWX},
where they improve
the $L^p-$ bounds on the control  from $L^8-$ bound
\eqref{eq:1.17}
to the following  $L^4-$ bound
 \begin{equation} \label{eq:1.16}
  \mathbb E\bigg[\bigg(\int_0^T|u(t)|^2
dt\bigg)^{2}\bigg]<\infty.
\end{equation}

 This paper is concerned with the partial information optimal control problem
where the  system is governed by a controlled forward-backward stochastic differential equation of jump diffusion system
 with correlated noises between the system and the observation.
 The main contribution of this paper is to establish  necessary and
sufficient stochastic maximum  for an optimal control in a unified way. The
main idea is to get directly a variation formula in terms of the
Hamiltonian and the associated adjoint system which is a linear
forward-backward
stochastic differential equation  and neither the
variational systems nor the corresponding Taylor type expansions of
the state process and the cost functional will be considered.
Moreover, different from \eqref{eq:1.16},  the following  $L^4-$ bound
is imposed on our admissible controls:
 \begin{equation} \label{eq:1.1116}
  \mathbb E\bigg[\int_0^T|u(t)|^4
dt\bigg]<\infty,
\end{equation}
because the BDG-inequality for
 the integration of 
Poisson random martingale is different 
from  that for
 the integration of Brownian Motion.
 
 The rest of this paper is organized as follows. The assumptions, notations and
 the formulation of our partial
   observable optimal control problem are given in Section 2. Sections 3 and 4 are devoted to prove our main results. Finally, Section 5 concludes the paper and outlines some possible future developments.

\section{  Assumptions and Statement  of Problem}

In this section, we introduce
some basic notations  which will be
used in this paper. Let ${\cal T} : = [0, T]$ denote a fixed time interval of finite length, i.e., $T < \infty$.
We consider a complete probability space $( \Omega, {\mathscr F}, {\mathbb P} )$,
on which all randomness is defined. We equip $( \Omega, {\mathscr F}, {\mathbb P} )$
with a right-continuous, ${\mathbb P}$-complete filtration ${\mathbb F} : = \{ {\mathscr F}_t | t \in {\cal T} \}$,
to be specified below. Furthermore, we assume that ${\mathscr F} _{T} = {\mathscr F}$. Denote by
${\mathbb E} [\cdot]$ the expectation taken with respect to ${\mathbb P}$.  By $\mathscr{P}$ we
denote the
predictable $\sigma$ field on $\Omega\times {\cal T} $ associated with $\mathbb F$
and by $\mathscr
B(\Lambda)$ the Borel $\sigma$-algebra of any topological space
$\Lambda.$ Let
  $\{W(t), t \in {\cal T}\}$ and $\{Y(t),t \in {\cal T}\}$
be two independent one-dimensional
standard Brownian motions.
Let $\{\mathscr{F}^W_t\}_{t\in {\cal T}}$ and $
\{\mathscr{F}^Y_t\}_{t\in {\cal T}}$ be $\mathbb P$-completed natural
filtration generated by $\{W(t), t\in {\cal T}\}$ and $\{Y(t), t\in {\cal T}\},$ respectively.  By $\mathscr{P}^Y$ we
denote the
predictable $\sigma$ field on $\Omega\times {\cal T} $ associated with $
\{\mathscr{F}^Y_t\}_{t\in {\cal T}}.$
Let $(E,\mathscr B (E), \nu
)$ be a measurable space with $\nu(E)<\infty $ and $\eta: \Omega\times D_\eta \longrightarrow E$ be an
$\mathscr F_t$-adapted stationary
Poisson point process with characteristic measure $\nu$, where
 $D_\eta$  is a countable subset of $(0, \infty)$. Then the counting measure induced by $\eta$ is
$$
\mu((0,t]\times A):=\#\{s\in D_\eta; s\leq t, \eta(s)\in A\},~~~for~~~ t>0, A\in \mathscr B (E).
$$
And $\tilde{\mu}(dt, d\theta):=\mu(dt, d\theta)-dt\nu(d\theta)$ is a compensated Poisson random martingale  measure which
is assumed to be independent of Brownian motion
$\{W(t), 0\leq t\leq T\}$ and  $\{Y(t), 0\leq t\leq T\}$. Assume $\{\mathscr{F}_t\}_{0\leq t\leq T}$ is the  $\mathbb P$-completed natural
 filtration generated by $\{W(t),
0\leq t\leq T\}, \{Y(t),
0\leq t\leq T\} $ and $\{\iint_{(0,t] \times A }\tilde{\mu}(d\theta,ds), 0\leq t\leq T, A\in \mathscr B (E) \}.$
 Let $H$ be a Euclidean space. The inner product in $H$ is denoted by
$\langle\cdot, \cdot\rangle,$ and the norm in $ E$ is denoted by $|\cdot|.$
Let $A^{\top }$ denote the
transpose of the matrix or vector $A.$
For a
function $\psi:\mathbb R^n\longrightarrow \mathbb R,$ denote by
$\psi_x$ its gradient. If $\psi: \mathbb R^n\longrightarrow \mathbb R^k$ (with
$k\geq 2),$ then $\psi_x=(\frac{\partial \phi_i}{\partial x_j})$ is
the corresponding $k\times n$-Jacobian matrix. In the follows, $K$ represents a generic constant, which
can be different from line to line.
Next we introduce some spaces of random variable and stochastic
 processes.
For any $\alpha, \beta\in [1,\infty),$ let
 \begin{enumerate}
\item[$\bullet$]
$M_{\mathscr{F}}^\beta(0,T;H)$: the space of all $E$-valued and $\mathscr P-$measurable processes $f=\{f(t,\omega),\ (t,\omega)\in {\cal T}
\times\Omega\}$ satisfying
$
\|f\|_{M_{\mathscr{F}}^\beta(0,T;H)}\triangleq{\left (\mathbb E\bigg[\displaystyle%
\int_0^T|f(t)|^ \beta dt\bigg]\right)^{\frac{1}{\beta}}}<\infty, $
 \item[$\bullet$] $S_{\mathscr{F}}^\beta (0,T;H):$ the space of all $H$-valued and ${
\mathscr{F}}_t$-adapted c\`{a}dl\`{a}g processes $f=\{f(t,\omega),\
(t,\omega)\in {\cal T}\times\Omega\}$ satisfying $
\|f\|_{S_{\mathscr{F}}^\beta(0,T;H)}\triangleq{\left (\mathbb E\bigg[\displaystyle\sup_{t\in {\cal T}}|f(t)|^\beta \bigg]\right)^{\frac
{1}{\beta}}}<+\infty,$
 \item[$\bullet$]$L^\beta (\Omega,{\mathscr{F}},P;H):$ the space of all
$H$-valued random variables $\xi$ on $(\Omega,{\mathscr{F}},P)$
satisfying $ \|\xi\|_{L^\beta(\Omega,{\mathscr{F}},P;H)}\triangleq
\sqrt{\mathbb E|\xi|^\beta}<\infty,$
 \item[$\bullet$] $M_{\mathscr{F}}^\beta(\Omega;L^\alpha (0,T; H)):$ the space of all $L^\alpha (0,T; H)$-valued and $\mathscr P-$measurable processes $f=\{f(t,\omega),\ (t,\omega)\in[0,T]%
\times\Omega\}$ satisfying $
\|f\|_{M_{\mathscr{F}}^\beta(\Omega;L^\alpha (0,T; H))}\triangleq{\left\{\mathbb E\bigg[\left(\displaystyle
\int_0^T|f(t)|^\alpha
dt\right)^{\frac{\beta}{\alpha}}
\bigg]\right\}^{\frac{1}{\beta}}}<\infty, $

\item[$\bullet$] ${M}^{\nu,\beta}( E; H):$ the space of all H-valued measurable
  functions $\Lambda=\{\Lambda(e), e \in E\}$ defined on the measure space $(E, \mathscr B(E); \nu)$ satisfying
$$\|\Lambda\|_{{ M}^{\nu,\beta}( E; H)}\triangleq\bigg\{{\displaystyle
\int_E|\Lambda(e)|^\beta\nu(de)}\bigg\}^
{\frac{1}{\beta}}<~\infty,$$

\item[$\bullet$] ${M}^{\nu,\beta}( [0,T]\times E; H):$ the space of all H-valued measurable
  functions $\Lambda=\{\Lambda(t,e), (t,e) \in {\cal T}\times E\}$ defined on the measure space $({\cal T}\times E, \mathscr B({\cal T})\times \mathscr B(E); (dt)\times \nu)$ satisfying
$$\|\Lambda\|_{{M}^{\nu,\beta}( [0,T]\times E; H)}\triangleq
\bigg\{{\displaystyle
\int_0^T\int_E|\Lambda(t,e)|^\beta\nu(de)dt}
\bigg\}^{\frac{1}{\beta}}<~\infty,$$
\item $\bullet$ ${M}_{\mathscr{F}}^{\nu,\beta}{([0,T]\times  E; H)}:$ the  space of all $H$-valued
and $\mathscr P\times \mathscr B(E)-$measurable processes $\Lambda=\{\Lambda(t,\omega,e),\
(t,\omega,e)\in[0,T]\times\Omega\times E\}$ satisfying
$$\|\Lambda\|_{{M}_{\mathscr F}^{\nu,\beta}( [0,T]\times E; H)}\triangleq \bigg\{{\mathbb E\bigg[\int_0^T\int_{E}\displaystyle|\Lambda(t,e)|
^\beta dt}\bigg]\bigg\}^{\frac{1}{\beta}}<~\infty,$$
\item[$\bullet$] $M_{\mathscr{F}}^\beta(\Omega;
    {M}^{\nu,\alpha}{([0,T]\times  E; H)}):$ the space of all $M_
    {\mathscr{F}}^{\nu,\alpha}{([0,T]\times  E; H)}$-valued and$\mathscr P\times \mathscr B(E)-$measurable processes $\Lambda=\{\Lambda(t,\omega,e),\ (t,\omega,e)\in[0,T]%
\times\Omega\times E$ satisfying $
\|\Lambda\|_{\alpha,\beta}\triangleq{\left\{\mathbb E\bigg[\left(\displaystyle
\int_0^T\int_{E}|\Lambda(t,e)|^\alpha
\nu(de)dt\right)^{\frac{\beta}
{\alpha}}\bigg]\right\}^{\frac{1}{\beta}}}<\infty. $
\end{enumerate}

Now we consider  the following
forward-backward stochastic
control system of Jump diffusion
\begin{eqnarray} \label{eq:1}
\left\{
\begin{aligned}
dx(t)=&b(t,x(t),u(t))dt+ \sigma_1(t, x(t), u(t)) dW(t)+\sigma _2(t, x(t), u(t)) dW^u(t)
+\int_{E}g (t, x(t-), u(t),e) \tilde\mu(dt,de),
\\
dy(t)=&f(t,x(t),y(t),z_1(t),z_2(t),\Lambda(t,\cdot),u(t))dt+ z_1(t) dW(t)+  z_2(t) dW^u(t)+
\int_{E}\Lambda(t,e)\tilde\mu(dt,de),\\
x(0)=&x,\\
y(T)=&\phi(x(T))
\end{aligned}
\right.
\end{eqnarray}
with one observation  processes $Y(\cdot)$ governed  by the following SDE

\begin{eqnarray}\label{eq:2}
\left\{
\begin{aligned}
dY(t)=&h(t,x(t),u(t))dt+ dW^u(t),\\
Y(t)=& 0.
\end{aligned}
\right.
\end{eqnarray}
 Here $b: {\cal T} \times \Omega \times {\mathbb R}^n
 \times U  \rightarrow {\mathbb R}^n$,
  $\sigma_1:
{\cal T} \times \Omega \times {\mathbb R}^n
 \times U \rightarrow {\mathbb R}^n$, $\sigma_2: {\cal T} \times \Omega \times {\mathbb R}^n
 \times U \times E \rightarrow {\mathbb R}^n $,
 $f: {\cal T} \times \Omega \times {\mathbb R}^n \times {\mathbb R}^m\times {\mathbb R}^m \times {\mathbb R}^m\times {M}^{\nu,2}( E; \mathbb R^m)
 \times U \rightarrow {\mathbb R}^m,
 \phi: \Omega \times {\mathbb R}^n  \rightarrow {\mathbb R^m}$
 and $h: {\cal T} \times \Omega \times {\mathbb R}^n
 \times U \rightarrow {\mathbb R}$  are given random mapping with
 $U$ being a nonempty convex subset of $\mathbb R^k.$ In the above equations, $u(\cdot)$ is our admissible control
 process defined as follows.

\begin{definition}
   A  stochastic process $u(\cdot)$ is
    said to be an admissible control
    process if it is
   a  ${\mathscr P}^Y$-measurable process valued in an nonempty
  convex subset $U$  in  $\mathbb R^K$ satisfying  $$\mathbb E\bigg[\int_0^T|u(t)|^4
dt\bigg]<\infty.$$
   We denote the set of all admissible controls by $\cal A.$
\end{definition}

The following standard  assumptions
 are imposed on the coefficients of the equations \eqref{eq:1} and
 \eqref{eq:2}.

\begin{ass}\label{ass:1.1}
(i)The mappings $b$, $\sigma_1,\sigma_2 $ and $h$  are ${\mathscr P} \otimes {\mathscr
B} ({\mathbb R}^n)  \otimes {\mathscr B}
(U) $-measurable,   $g$  is  ${\mathscr P} \otimes {\mathscr
B} ({\mathbb R}^n)  \otimes {\mathscr B}
(U) \otimes {\mathscr B}
(E) $-measurable. For almost all $(t, \omega)\in {\cal T}
\times \Omega$, the mappings $b,\sigma_1,\sigma_2, h$ and $g$
are of  appropriate growths with respect to
$(x, u),$ i.e., there exist a constant $C$ and
a deterministic function $ C(e)$
such that  for  all $x\in \mathbb
R^n, u\in U$ and a.e. $(t, \omega)\in {\cal T} \times \Omega,$
\begin{eqnarray*}
\left\{
\begin{aligned}
& |\alpha(t,x,u)|
\leq C (1+|x|+|u|), \alpha=b, \sigma_1,
\\&|\beta(t,x,u)|\leq
C, \beta=h, \sigma_2,
\\& |g(t,x,u,e)|\leq C(e)(1+|x|+|u|),
 \\& \int_{E}|C(e)|^4\nu(de)<\infty.
\end{aligned}
\right.
\end{eqnarray*}
Moreover, for all $(t, \omega, e) \in {\cal T} \times \Omega\times E$,  $b$, $\sigma_1, \sigma_2, h$, and $g$ are continuous  differentiable with respect to
$(x, u)$ and  the corresponding derivatives
are continuous and uniformly bounded, i.e.,
there exist a constant $L$ and
a deterministic function $L(e)$
 such that for  all $x\in \mathbb
R^n, u\in U$ and a.e. $(t, \omega)\in {\cal T} \times \Omega,$
\begin{eqnarray*}
\left\{
\begin{aligned}
& |\alpha_x(t,x,u)|
+|\alpha_u(t,x,u)|
\leq L, \alpha=b, \sigma_1, \sigma_2, h,
\\&|g_{x}(t,x,u,e)|
+|g_{u}(t,x,u,e)|\leq
L(e),
\\&\int_{E}|L(e)|^4\nu(de)< \infty.
\end{aligned}
\right.
\end{eqnarray*}
(ii)The coefficient $f$ is  ${\mathscr P} \otimes {\mathscr
B} ({\mathbb R}^n)\otimes {\mathscr
B} ({\mathbb R}^m)\otimes {\mathscr
B} ({\mathbb R}^m) \otimes {\mathscr
B} ({\mathbb R}^m)\otimes  {\mathscr B} ({M}^{\nu,2}( E; \mathbb R^m))\otimes {\mathscr B}(U) $-measurable
and $f(\cdot,0,0,0,0,0,0)\in M_{\mathscr{F}}^4(0,T;L^2 (0,T; \mathbb R^m))$. For almost all $(t, \omega)\in {\cal T}
\times \Omega$, the mapping $f$
is  G\^ateaux differentiable
with respect to $(x,y,z_1,z_2, \Lambda(\cdot),u) \in \mathbb {R}^n\times \mathbb  R^m\times \mathbb R^m\times \mathbb R^m \times {M}^{\nu,2}( E; \mathbb R^m)\times
U$  with continuous and uniformly bounded G\^ateaux derivatives.\\
(iii) The coefficient
$\phi$ is ${\mathscr F}_T \otimes {\mathscr B} ({\mathbb
R}^n) $-measurable.  For almost all $(t, \omega)\in [0,T]
\times \Omega$, the mapping $\phi$
 is continuous differentiable with respect to $x$ with
appropriate growths.
More precisely, there exists a constant $C
> 0$  such that for  all $x\in \mathbb {R}^n$ and a.e. $ \omega\in  \Omega,$

\begin{eqnarray}
\begin{split}
& (1+|x|)^{-1}
|\phi(x)|
+|\phi_x(x)|
\leq C.
\end{split}
\end{eqnarray}

\end{ass}

Now we
begin to discuss  the well- posedness
of \eqref{eq:1} and \eqref{eq:2}.
Indeed, putting \eqref{eq:2} into the state equation \eqref{eq:1}, we get that
\begin{eqnarray} \label{eq:4}
\left\{
\begin{aligned}
dx(t)=&(b-\sigma_2h)(t,x(t),u(t))dt+ \sigma_1(t, x(t), u(t)) dW(t)+\sigma _2(t, x(t), u(t)) dY(t)+\int_{E}g (t, x(t-), u(t),e) \tilde\mu(dt,de),
\\
dy(t)=&(f(t,x(t),y(t),z_1(t),z_2(t),\Lambda(t,\cdot), u(t))
-z_2(t)h(t,x(t),u(t)))dt+ z_1(t) dW(t)+  z_2(t) dY(t)+\int_{E}\Lambda(t,e) \tilde\mu(dt,de),\\
x(0)=&x,\\
y(T)=&\Phi(x(T)).
\end{aligned}
\right.
\end{eqnarray}
Under Assumption \ref{ass:1.1}, for any
admissible control $u(\cdot)\in \cal A,$
we have the following basic
result.

\begin{lemma}\label{lem:3.3}
  Let Assumption \ref{ass:1.1}
  be satisfied. Then for any
  admissible control $u(\cdot)\in
  \cal A,$  the  equation \eqref{eq:4} have a unique strong
  solution $( x(\cdot),  y(\cdot), z_1(\cdot), z_2(\cdot), \Lambda(\cdot,\cdot))\in S_{\mathscr{F}}^4 ([0, T];\mathbb R^{n})\times S_{\mathscr{F}}^4 (0,T;\mathbb R^{m})
  \times M_{\mathscr{F}}^4(\Omega;L^2 (0,T; \mathbb R^m))\times M_{\mathscr{F}}^4(\Omega;L^2 (0,T; \mathbb R^m))\times M_{\mathscr{F}}^4(\Omega;
  {M}^{\nu,2}{([0,T]\times  E; \mathbb R^m)}).$
Moreover, we have the following
estimate:

\begin{eqnarray}\label{eq:1.8}
\begin{split}
 &{\mathbb E} \bigg [ \sup_{t\in \cal T} | x(t) |^4 \bigg
] +{\mathbb E} \bigg [ \sup_{t\in \cal T} | y(t) |^4 \bigg
]+{\mathbb E} \bigg [ \bigg(\int_0^T | z_1(t) |^2 dt\bigg)^{2}\bigg]
+{\mathbb E} \bigg [ \bigg(\int_0^T | z_2(t) |^2 dt\bigg)^{2}\bigg
]+{\mathbb E} \bigg [ \bigg(\int_0^T
 \int_{E}| \Lambda(t,e) |^2 \nu(de)dt\bigg)^{2}\bigg
]
\\&\leq K\bigg \{1+|x|^4+\mathbb E\bigg[\Big(\int_0^T
|f(t,0,0,0,0,0,0)|^2dt\Big)^{2}\bigg]+\mathbb E\bigg[\int_0^T |u(t)|^4dt\bigg] \bigg\}.
\end{split}
\end{eqnarray}
Further, if $(\bar x(\cdot), \bar y(\cdot), \bar z_1(\cdot),
\bar z_2(\cdot),\bar\Lambda(\cdot))$ is the unique strong solution
corresponding to another admissible control  $ \bar u(\cdot)\in \cal A,$ then
we have  the following
estimate:
\begin{eqnarray}\label{eq:1.15}
\begin{split}
&{\mathbb E} \bigg [ \sup_{t\in \cal T} | x (t) - \bar x(t) |^4
\bigg ]+{\mathbb E} \bigg [ \sup_{t\in \cal T} | y (t) - \bar y(t) |^4
\bigg ]+{\mathbb E} \bigg [ \bigg(\int_0^T | z_1(t)-\bar z_1(t) |^2 dt\bigg)^{2}\bigg
]+{\mathbb E} \bigg [ \bigg(\int_0^T | z_2(t)-\bar z_2(t) |^2 dt\bigg)^{2}\bigg
]
\\&+{\mathbb E} \bigg [ \bigg(\int_0^T
 \int_{E}| \Lambda(t,e) -\bar \Lambda(t,e)|^2 \nu(de)dt\bigg)^{2}\bigg
]
\\&\leq  K {\mathbb E} \bigg [ \int_0^T| u(t)- \bar u (t)
|^4 dt\bigg ].
\end{split}
\end{eqnarray}
\end{lemma}
\begin{proof}
  The proof can be directly obtained
  by  combining Lemma A.3 in
  \cite{FT2015} and Theorem 4.1 in
  \cite{CT2010}.
\end{proof}

For  the strong solution
 $( x^u(\cdot),  y^u(\cdot), z_1^u(\cdot), z_2^u(\cdot), \Lambda^u(\cdot, \cdot))$ of  the equation
 \eqref{eq:4} associated with
 any given admissible control
$u(\cdot)\in \cal A,$ we  introduce a process
\begin{eqnarray}\label{eq:7}
\begin{split}
  \rho^u(t)
  =\displaystyle
  \exp^{\bigg\{\displaystyle\int_0^t h(s, x^u(s), u(s))dY(s)
  -\frac{1}{2} h^2(s, x^u(s), u(s))ds\bigg\}},
  \end{split}
\end{eqnarray}
which is abviously the solution to the following  SDE
\begin{eqnarray} \label{eq:8}
  \left\{
\begin{aligned}
  d \rho^u(t)=&  \rho^u(t) h(s, x^u(s), u(s))dY(s)\\
  \rho^u(0)=&1.
\end{aligned}
\right.
\end{eqnarray}

 For the stochastic process
 $\rho^u(\cdot),$  we have the
 following  basic result.
 \begin{lemma}\label{lem:3.4}
  Let Assumption \ref{ass:1.1} holds. Then for any $u(\cdot)\in \cal A,$ we have
 for any $\alpha \geq 2,$
\begin{eqnarray}\label{eq:1.9}
\begin{split}
 {\mathbb E} \bigg [ \sup_{{t\in
 \cal T}} | \rho^u(t) |^\alpha \bigg
] \leq K.
\end{split}
\end{eqnarray}
Further, if $ \bar \rho(\cdot)$ is the
process defined by \eqref{eq:7}
or \eqref{eq:8}
corresponding to another
 admissible control $ \bar u(\cdot)\in \cal A,$  then the following
estimate holds
\begin{eqnarray}\label{eq:1.15}
{\mathbb E} \bigg [ \sup_{t\in \cal T} |
\rho^u (t) - \bar \rho (t) |^2
\bigg ]  \leq  K \bigg\{{\mathbb E} \bigg [ \int_0^T| u(t)- \bar u (t)
|^2 dt\bigg ]^{{2}}\bigg\}^{\frac{1}{2}}.
\end{eqnarray}
\end{lemma}

\begin{proof}
  The proof can be directly obtained
  by  combining Proposition 2.1 in
  \cite{Mou}.
\end{proof}

Under Assumption \ref{ass:1.1},
$\rho^u(\cdot)$ is
an  $( \Omega,
{\mathscr F}, \{\mathscr{F}_t\}_{t\in {\cal T}}, {\mathbb P} )-$
martingale. Define a new probability measure $\mathbb P^u$ on $(\Omega, \mathscr F)$ by
\begin{eqnarray}
  d\mathbb P^u=\rho^u(1)d\mathbb P.
\end{eqnarray}
 Then from Girsanov's theorem and \eqref{eq:2}, $(W(\cdot),W^u(\cdot))$ is an
$\mathbb R^2$-valued standard Brownian motion defined in the new probability
space $(\Omega, \mathscr F, \{\mathscr{F}_t\}_{0\leq t\leq T},\mathbb P^u).$
So $(\mathbb P^u, x^u(\cdot), y^u(\cdot),
z^u_1(\cdot), z^u_2(\cdot), \Lambda^u(\cdot,\cdot), \rho^u(\cdot),
 W(\cdot), W^u(\cdot))$ is a weak
solution on $(\Omega, \mathscr F, \{\mathscr{F}_t\}_{t\in \cal
T})$ of  \eqref{eq:1} and
\eqref{eq:2}.

 The cost functional is given by
\begin{eqnarray}\label{eq:13}
  \begin{split}
    J(u(\cdot)=\mathbb E^u\bigg[\int_0^Tl(t,x(t),y(t),z_1(t),z_2(t),
    \Lambda(t,\cdot), u(t))dt+ \Phi(x(T))+\gamma(y(0))\bigg].
  \end{split}
\end{eqnarray}
 where $\mathbb E^u$ denotes the expectation with respect to the
probability space $(\Omega, \mathscr F, \{\mathscr{F}_t\}_{0\leq
t\leq T},\mathbb P^u)$ and  $l:
{\cal T} \times \Omega \times {\mathbb R}^n \times {\mathbb R}^m
\times {\mathbb R}^m\times {\mathbb R}^m
\times {M}^{\nu,2}( E; \mathbb R^m)
 \times U \rightarrow {\mathbb R},$ $\Phi: \Omega \times {\mathbb R}^n  \rightarrow {\mathbb R}$
 and  $\gamma: \Omega \times {\mathbb R}^m \rightarrow {\mathbb R}$
  are given random mappings
satisfying  the following assumption:

 \begin{ass}\label{ass:1.2}
 $l$ is ${\mathscr P} \otimes {\mathscr
B} ({\mathbb R}^n) \otimes {\mathscr B} ({\mathbb R}^m)\otimes {\mathscr B} ({\mathbb R}^m) \otimes {\mathscr B} ({\mathbb R}^m)
\otimes {\mathscr B}({M}^{\nu,2}( E; \mathbb R^m))\otimes {\mathscr B}
(U) $-measurable,  $\Phi$ is ${\mathscr F}_T \otimes {\mathscr B} ({\mathbb
R}^n) $-measurable, and $\gamma$ is ${
\mathscr F}_0 \otimes {\mathscr B} ({\mathbb
R}^n) $-measurable.    For almost all $(t, \omega)\in [0,T]
\times \Omega$, the mappings
\begin{eqnarray*}
(x,y,z_1, z_2,\Lambda(\cdot), u) \rightarrow l(t,\omega,x, y, z_1,z_2, \Lambda(\cdot), u),
\end{eqnarray*}
\begin{eqnarray*}
x \rightarrow \Phi(\omega,x)
\end{eqnarray*}
and
\begin{eqnarray*}
y\rightarrow \gamma(\omega,y)
\end{eqnarray*}
are continuous G\^ateaux differentiable
 with respect to $(x, y,z_1, z_2, \Lambda(\cdot),u)$ with
appropriate growths, respectively.
More precisely, there exists a constant $C
> 0$  such that for  all $(x,y, z_1,z_2, \Lambda(\cdot),
u) \in \mathbb {R}^n\times \mathbb R^m
\times \mathbb R^m\times \mathbb R^m
\times {M}^{\nu,2}( E; \mathbb R^m)\times U$ and a.e. $(t, \omega)\in [0,T]
\times \Omega,$
\begin{eqnarray*}
\left\{
\begin{aligned}
&
(1+|x|+|y|+|z_1|+|z_2|+||\Lambda(\cdot)||_{{M}^{\nu,2}( E; \mathbb R^m)}+|u|)^{-1} (|l_x(t,x,y,z_1,z_2,\Lambda(\cdot), u)|+|l_y(t,x,y,z_1,z_2,\Lambda(\cdot), u)|
\\&+
|l_{z_1}(t,x,y,z_1,z_2,\Lambda(\cdot), u)|
+|l_{z_2}(t,x,y,z_1,z_2,
\Lambda(\cdot),u)|
+|l_u(t,x,y,z_1,z_2,\Lambda(\cdot),u)|+
||l_\Lambda(t,x,y,z_1,z_2,\Lambda(\cdot),u)
||_{{M}^{\nu,2}( E; \mathbb R^m)}
\\&+(1+|x|^2+|y|^2+|z_1|^2
+|z_2|^2+||\Lambda(\cdot)||^2_{{M}^{\nu,2}( E; \mathbb R^m)}
+|u|^2)^{-1}|l(t,x,y,z_1,z_2,\Lambda(\cdot),u)|
 \leq C;
\\
& (1+|x|^2)^{-1}|\Phi(x)| +(1+|x|)^{-1}|\Phi_x(x)|\leq
C;
\\
& (1+|y|^2)^{-1}|\gamma(y)| +(1+|y|)^{-1}|\gamma_y(y)|\leq
C.
\end{aligned}
\right.
\end{eqnarray*}
\end{ass}
 Under  Assumption \ref{ass:1.1} and
 \ref{ass:1.2},
 by the estimates \eqref{eq:1.8} and
 \eqref{eq:1.9},  we get that

\begin{eqnarray}
  \begin{split}
    |J(u(\cdot))|\leq & K  \bigg\{\mathbb E\bigg[\sup_{t\in {\cal T}}
    |\rho^u(t)|^2\bigg]\bigg\}^{\frac{1}{2}}
    \bigg\{ \mathbb E\bigg[\sup_{
    {t\in \cal T}}
    |x(t)|^4\bigg]+\mathbb E\bigg[\sup_{
    {t\in \cal T}}
    |y(t)|^4\bigg]
+\mathbb E\bigg[\bigg(\int_0^T|z_1(t)|^2
dt\bigg)^{2}\bigg]
\\&\quad\quad+\mathbb E\bigg[\bigg(\int_0^T|z_2(t)|^2
dt\bigg)^{2}\bigg]+{\mathbb E} \bigg [ \bigg(\int_0^T
 \int_{E}| \Lambda(t,e) |^2 \nu(de)dt\bigg)^{2}\bigg
]
+\mathbb E\bigg[\int_0^T|u(t)|^4
dt\bigg] +1\bigg\}^{\frac{1}{2}}
    \\  <& \infty,
  \end{split}
\end{eqnarray}
which implies that
the cost functional is well-defined.

Then we  can put forward the following partially observed optimal control problem  in its weak formulation,
 i.e., with changing
 the reference probability space $(\Omega, \mathscr F, \{\mathscr{F}_t\}_{0\leq
t\leq T},\mathbb P^u),$ as follows.

\begin{pro}
\label{pro:1.1} Find an admissible control $\bar{u}(\cdot)\in \cal$ such that
\begin{equation*}  \label{eq:b7}
J(\bar{u}(\cdot))=\displaystyle\inf_{u(\cdot)
\in \cal A}J(u(\cdot)),
\end{equation*}
subject to the
 state equation \eqref{eq:1}, the
 observation equation \eqref{eq:2}
 and the cost functional \eqref{eq:13}.

\end{pro}
Obviously, according to  Bayes' formula,
 the cost functional \eqref{eq:13} can be rewritten as
\begin{equation}\label{eq:15}
\begin{split}
J(u(\cdot ))=& \mathbb E\displaystyle\bigg[%
\int_{0}^{T}\rho^u(t)l(t,x(t),y(t),z_1(t),z_2(t),
\Lambda(t,\cdot), u(t))dt +\rho^u(T)\Phi(x(T))
+\gamma(y(0))\bigg].
\end{split}
\end{equation}
  Therefore, we  can translate   Problem \ref{pro:1.1}
  into  the following  equivalent optimal control problem in
  its strong formulation, i.e., without changing the  reference
probability space $(\Omega, \mathscr F, \{\mathscr{F}_t\}_{0\leq t\leq T},\mathbb P),$
where $\rho^u(\cdot)$ will be regarded as
an additional  state process besides the
state process $(x^u(\cdot), y^u(\cdot),
z^u_1(\cdot), z^u_2(\cdot),\Lambda(\cdot,\cdot)).$

\begin{pro}
\label{pro:1.2} Find an admissible control $\bar{u}(\cdot)$ such that
\begin{equation*}  \label{eq:b7}
J(\bar{u}(\cdot))=\displaystyle\inf_{u(\cdot)
\in \cal A}J(u(\cdot)),
\end{equation*}
 subject to
 the cost functional \eqref{eq:15}
 and  the following
 state equation
\begin{equation}
\displaystyle\left\{
\begin{array}{lll}
dx(t)=&(b-\sigma_2h)(t,x(t),u(t))dt+ \sigma_1(t, x(t), u(t)) dW(t)+\sigma _2(t, x(t), u(t)) dY(t)+\int_{E}g (t, x(t-), u(t),e) \tilde \mu(dt,de),
\\
dy(t)=&(f(t,x(t),y(t),z_1(t),z_2(t),\Lambda(t,\cdot),
u(t))
-z_2(t)h(t,x(t),u(t)))dt+ z_1(t) dW(t)+  z_2(t) dY(t)+\int_{E}\Lambda(t,e) \tilde \mu(dt,de),\\
d \rho^u(t)=&  \rho^u(t) h(s, x^u(s), u(s))dY(s),\\
  \rho^u(0)=&1,
\\
x(0)=&x,\\
y(T)=&\Phi(x(T)).
\end{array}%
\right.  \label{eq:3.7}
\end{equation}
\end{pro}
Any $\bar{u}(\cdot)\in \cal A$ satisfying above is called an optimal
control process of Problem \ref{pro:1.2} and the corresponding state
process $(\bar x(\cdot),
\bar y(\cdot),
\bar z_1(\cdot), \bar z_2(\cdot),
\bar \Lambda(\cdot,\cdot),
\bar \rho(\cdot))$ is called the optimal
state process. Correspondingly $(\bar{u}(\cdot);\bar x(\cdot),
\bar y(\cdot),
\bar z_1(\cdot), \bar z_2(\cdot),
\bar \Lambda(\cdot,\cdot),
\bar \rho(\cdot))$ is called an optimal pair of Problem \ref{pro:1.2}.
\begin{rmk}
 The present formulation
of the partially observed optimal control problem is quite similar
to a completely observed optimal control problem; the only
difference lies in the admissible class $\cal A$ of controls.
\end{rmk}

\section{Variation Calculus of The Cost Functional }

This purpose of this section 
is to  give a variation calculus of th cost Functional by Hamiltonian and 
adjoint processes. To this end, 
for the state equation \eqref{eq:3.7},
we first define  the
corresponding adjoint equation. Introduce  the Hamiltonian
 ${\cal H}: \Omega \times {\cal T} \times \mathbb R^n \times \mathbb R^m \times \mathbb R^{m}
  \times \mathbb R^{m}\times {M}^{\nu,2}( E; \mathbb R^m)\times U\times
\mathbb R^n\times \mathbb R^n
\times \mathbb R^n
  \times \mathbb R^{m}
  \times \mathbb R\rightarrow \mathbb R$  by
\begin{eqnarray}\label{eq6}
&& {\cal H} (t, x, y, z_1, z_2, \Lambda(\cdot), u, p,q_1, q_2, q_3(\cdot), k, R_2) \nonumber \\
&& = l (t, x, y, z_1,z_2, \Lambda(\cdot), u)
+ \langle b (t, x, u),  p \rangle
+ \langle \sigma_1(t, x, u), q_1\rangle + \langle \sigma_2 (t, x, u), q_2\rangle
+ \int_{E}\langle g (t, x, u,e), q_3(e)\rangle
\nu(de) \nonumber
\\&&\quad +
\langle f(t,x,y,z_1,z_2,\Lambda(\cdot),u), k\rangle  +\langle R_2, h(t,x,u)\rangle \ .
\end{eqnarray}

For any given admissible control pair $(\bar u(\cdot); \bar x(\cdot),
 \bar y(\cdot), \bar z_1(\cdot),
 \bar z_2(\cdot),\Lambda(\cdot,\cdot)), $ we define the corresponding adjoint equation by

\begin{numcases}{}\label{eq:18}
\begin{split}
d\bar r(t)&=- l(t,\bar x(t),
 \bar y(t), \bar z_1(t),
 \bar z_2(t),\bar\Lambda(t,\cdot), \bar u(t))dt
 +\bar R_1\left(
t\right)  dW\left(  t\right) +{\bar R}_{2}\left( t\right) dW^{
\bar u}\left( t\right)+\int_{E}\bar R_3(t,e) \tilde \mu(dt,de),
\\
d\bar p\left(  t\right)
 &=-{\cal {H}}_{x}\left( t,\bar x(t),
 \bar y(t), \bar z_1(t),
 \bar z_2(t),\bar \Lambda(t,\cdot) \bar u(t) \right)dt
 +\bar q_1
\left(  t\right)  dW\left(  t\right)  +{\bar q}_{2}\left( t\right) dW^{
\bar u}\left( t\right)+\int_{E}\bar q_3(t,e) \tilde \mu(dt,de),
\\
d\bar k\left(  t\right)  &=-{\cal {H}}_{y}\left(  t,\bar x(t),
 \bar y(t), \bar z_1(t),
 \bar z_2(t), \bar \Lambda(t,\cdot), \bar u(t) \right)dt-
{\cal {H}}_{z_1}\left(  t,\bar x(t),
 \bar y(t), \bar z_1(t),
 \bar z_2(t),\Lambda(t,\cdot), \bar u(t) \right)  dW\left(  t\right)
  \\&\quad\quad-{\cal {H}}_{z_2}\left(  t,\bar x(t),
 \bar y(t), \bar z_1(t),
 \bar z_2(t), \bar\Lambda(t,\cdot), \bar u(t) \right)dW^{
u}\left( t\right)-\int_{E} {\cal {H}}_{\Lambda}\left(  t,\bar x(t-),
 \bar y(t-), \bar z_1(t),
 \bar z_2(t), \bar\Lambda(t,\cdot) \bar u(t) \right)\tilde \mu(dt,de),
\\ \bar p(T)&=\Phi_x(\bar x(T))-\phi_x^*(\bar x(T))\bar k(T),
\\ \bar r(T)&=\Phi(\bar x(T)),
\\ \bar k(0)&=-\gamma_y(\bar y(0)),
\end{split}
\end{numcases}
where the following short-hand notation is used:
\begin{eqnarray}\label{eq:19}
\begin{split}
  &{\cal H}_a(t,\bar x(t),
 \bar y(t), \bar z_1(t),
 \bar z_2(t), \bar \Lambda(t,\cdot), \bar u(t))
 \\=&{\cal {H}}_{a}\left(  t,\bar x(t),
 \bar y(t), \bar z_1(t),
 \bar z_2(t),\bar \Lambda(t,\cdot), \bar u(t), \bar p(t), \bar q_1(t), \bar q_2(t),
 \bar q_3(t,\cdot),
  \bar k(t-), \bar R_2(t)-\bar\sigma_2^\top(t, x(t), u(t))\bar p(t)-\bar z_2 ^\top(t)
  \bar k(t) \right)
  \end{split}
\end{eqnarray}
where $a=x, y,z_1, z_2,\Lambda(\cdot),u.$

It is obvious that the adjoint equation \eqref{eq:18}
 a forward-backward
stochastic differential equation, and its solution
consists of  an 9-tuple process $(\bar p(\cdot),\bar q_1(\cdot), \bar { q}_2(\cdot),\bar q_3(\cdot,\cdot), \bar k(\cdot),r(\cdot),\bar R_1(\cdot),\bar R_2(\cdot ),
\bar R_3(\cdot,\cdot) ).$
By Assumptions \ref{ass:1.1} and
\ref{ass:1.2} and  by  Lemma A.3 in
  \cite{FT2015} and Theorem 4.1 in
  \cite{CT2010}, we know that  the adjoint equation \eqref{eq:18} have
a unique solution $(\bar p(\cdot),\bar q_1(\cdot), \bar { q}_2(\cdot),\bar { q}_3(\cdot,\cdot),
\bar k(\cdot), \bar r(\cdot),\bar R_1(\cdot),\bar R_2(\cdot ),\bar { R}_3(\cdot,\cdot) )\in S_{\mathscr{F}}^4(0,T;
\mathbb R^n)\times
M_{\mathscr{F}}^4(\Omega;L^2 (0,T; \mathbb R^n)) \times M_{\mathscr{F}}^4(\Omega;L^2 (0,T; \mathbb R^n)) \times M_{\mathscr{F}}^4(\Omega;
  {M}^{\nu,2}{([0,T]\times  E; \mathbb R^n)}) \times S_{\mathscr{F}}^4(0,T;
\mathbb R^m)\times S_{\mathscr{F}}^4(0,T;
\mathbb R)\times
M_{\mathscr{F}}^4(\Omega;L^2 (0,T; \mathbb R)) \times M_{\mathscr{F}}^4(\Omega;L^2 (0,T; \mathbb R)) \times M_{\mathscr{F}}^4(\Omega;
  {M}^{\nu,2}{([0,T]\times  E; R)}),
$  also said to be the adjoint process 
associated with
the admissible pair $(\bar{u}(\cdot);\bar x(\cdot),
\bar y(\cdot),
\bar z_1(\cdot), \bar z_2(\cdot),
\bar \Lambda(\cdot,\cdot),
\bar \rho(\cdot))$.

Suppose that $({u}(\cdot);x^u(\cdot),
 y^u(\cdot),
 z^u_1(\cdot),  z^u_2(\cdot),
 \Lambda^u(\cdot),
\rho^u(\cdot))$ and $(\bar{u}(\cdot);\bar x(\cdot),
\bar y(\cdot),
\bar z_1(\cdot), \bar z_2(\cdot),
\bar \Lambda(\cdot,\cdot),
\bar \rho(\cdot))$ are
two admissible pairs. In the following, we
will  give an formula for the difference $J (u (\cdot)) - J (\bar u (\cdot))$
 using the Hamiltonian ${\cal H}$ and the adjoint process $(\bar p(\cdot),\bar q_1(\cdot), \bar { q}_2(\cdot), \bar q_3(\cdot,\cdot),\bar k(\cdot),\bar R_1(\cdot),\bar R_2(\cdot ), \bar R_3(\cdot,\cdot)).$
as well as other relevant expressions.

To simplify  our notation,  the
following short-hand notations are introduced:
\begin{eqnarray}\label{eq:3.16}
\left\{
\begin{aligned}
& \gamma^u (0)
= \gamma( y^u (0)) \ ,
 \bar \gamma(0) =
 \gamma( \bar y (0)), \
 \\& \phi^u (T)
= \phi( x^u (T)),
 \bar \phi(T) =  \phi( \bar x (T)) \ ,
\\& \Phi^u (T)
= \Phi( x^u (T)),
 \bar \Phi(T) =  \Phi( \bar x (T)),\\
& \alpha^u (t) = \alpha ( t, x^u (t),u (t) ),\\
& \bar \alpha (t) = \alpha( t, \bar x (t), \bar
u(t) ), \quad \alpha =  b, \sigma_1,
\sigma_2, h,
\\
& \beta^u (t) = \alpha ( t, x^u (t),
y^u(t), z^u_1(t),z^u_2(t), \Lambda^u(t,\cdot),
u (t) ),\\
& \bar \beta (t) = \alpha( t, \bar x (t), \bar y(t), \bar z_1(t),\bar z_2(t), \bar\Lambda(t,\cdot), \bar
u(t) ), \quad \beta =  f, l,
\\
& g^u (t,e) = g( t, x^u (t-),u (t),e ),\\
& \bar g (t,e) = g( t, \bar x (t-), \bar
u(t),e ).
\end{aligned}
\right.
\end{eqnarray}

\begin{lemma}\label{lem4}
 Assume that Assumptions \ref{ass:1.1} and \ref{ass:1.2} hold.
By using the short-hand notations \eqref{eq:19} and
\eqref{eq:3.16},  it follows that
\begin{eqnarray}\label{eq:21}
&&J (u (\cdot)) - J (\bar u(\cdot))\nonumber
\\ &=& {\mathbb E} ^{\bar u}
\bigg[\int_0^T \bigg [ { \cal H}(t, x^u(t),
y^u(t), z^u_1(t), z^u_2(t), \Lambda^u(t,\cdot),  u(t)) - {\cal H} (t, \bar x(t),
\bar y(t), \bar z_1(t),\bar z_2(t), \bar\Lambda(t,\cdot),  \bar u(t)) \nonumber \\&& - \big < {\cal H}_x (t, \bar x(t),
\bar y(t), \bar z_1(t), \bar z_2(t), \bar\Lambda(t,\cdot),  \bar u(t)),x^u (t) - \bar
x (t)
\big>\nonumber
\\&&- \big <{ \cal H}_y (t, \bar x(t),
\bar y(t), \bar z_1(t), \bar z_2(t), \bar\Lambda(t,\cdot),  \bar u(t)),
y^u (t) - \bar
y (t)\big>\nonumber
\\&&- \big <{\cal H}_{z_1} (t, \bar x(t),
\bar y(t), \bar z_1(t),\bar z_2(t), \bar\Lambda(t,\cdot),  \bar u(t)),
 z^u_1 (t) - \bar
z_1 (t)\big>\nonumber
\\&&- \big <{\cal H}_{z_2} (t, \bar x(t),
\bar y(t), \bar z_1(t),\bar z_2(t), \bar\Lambda(t,\cdot),  \bar u(t)),
 z^u_2 (t) - \bar
z_2 (t)\big>
\nonumber
\\&&- \int_{E}\big <{\cal H}_{\Lambda} (t, \bar x(t),
\bar y(t), \bar z_1(t),\bar z_2(t), \bar\Lambda(t,\cdot),  \bar u(t)),
 \Lambda^u_2 (t,e) - \bar
\Lambda_2 (t,e)\big>\nu(de)
\nonumber
\\&&-
 \langle (\sigma_2^u(t)-\bar{\sigma}_2(t))
 (h^u(t)-\bar
 h(t)), \bar p(t)\rangle \nonumber
 \\&& -
 \langle (z_2^u(t)-\bar{z}_2(t))
 (h^u(t)-\bar
 h(t)), \bar k(t)\rangle dt\bigg]
\nonumber \\
&& + {\mathbb E}^{\bar u} \big [ \Phi^u(T)
 - \bar \Phi (T) - \left <
x^u (T) - \bar x (T), \bar \Phi_x (T)  \right
> \big ]
\nonumber \\
&& - {\mathbb E}^{\bar u}
\big [ \la \phi^u (T) - \bar \phi (T), \bar k(T)\ra - \left <
\bar \phi_x^*(T)\bar k(T),
x^u (T) - \bar x (T) \right
> \big ]
\nonumber \\
&& + {\mathbb E}
 \big [ \gamma^u(0) - \bar \gamma (0) - \left <
y^u (0)
- \bar y (0), \bar \gamma_y (0)  \right
> \big ]
\nonumber\nonumber
\\&&+\mathbb E\bigg[\int_0^T \bar R_2(t)(\rho^u(t)-\rho^{\bar u}(t))(h^u(t)-\bar
 h(t))dt\bigg] \nonumber
\\&&+\mathbb E\bigg[\int_0^T(l^u(t)-\bar l(t))( \rho^u(t)-\bar
\rho(t))dt\bigg]\nonumber
\\&&+\mathbb E \bigg[ ( \rho^u (T) - \bar \rho(T)) (\Phi^u (T) -
\bar \Phi (T))\bigg].
\end{eqnarray}
\end{lemma}

\begin{proof}
  In view of \eqref{eq:2}, we can easily check that $(x^u(\cdot), y^u(\cdot), z^u_1(\cdot), z^u_2(\cdot), \bar \Lambda(\cdot))$ admits the following FBSDE:
\begin{eqnarray}
\left\{
\begin{aligned}
dx(t)=&\big[b^{u}(t)+\sigma _2^u(t)(\bar  h(t)-h^u(t))\big]dt+
\sigma^u_1(t) dW(t)+\sigma _2^u(t) dW^{\bar u}(t)+\int_{E}g^u(t,e) \tilde \mu(dt,de)
\\
dy(t)=&\big[f^u(t)+z _2(t)( \bar h(t)- h^u(t))\big]dt+ z_1(t) dW(t)+  z_2(t) dW^{\bar u}(t)+\int_{E}\Lambda (t,e) \tilde \mu(dt,de)\\
x(0)=&x,\\
y(T)=&\phi(x(T)).
\end{aligned}
\right.
\end{eqnarray}
Therefore, we have that $(x^u(t)-\bar x(t),
y^u(t)-\bar y(t), z^u(t)-\bar z(t),
\Lambda^u(t,e)-\bar \Lambda(t,e))$ 
 admits the following FBSDE:

\begin{eqnarray}
\left\{
\begin{aligned}
dx(t)-\bar x(t)=&\big[b^u(t)-\bar b(t)+\sigma _2^u(t)( h^{\bar u}(t)-h^u(t))\big]dt+ [\sigma^u_1(t)-\bar\sigma_1(t)] dW(t)+[\sigma _2^u(t)-\bar\sigma _2(t)]dW^{\bar u}(t)\\&+\int_{E}[g^u (t,e)-\bar g(t,e)]\tilde \mu(dt,de)
\\
dy(t)-\bar y(t)=&\big[f^u(t)-\bar f(t) +z _2(t)( \bar h(t)-h(t))\big]dt+ [z_1(t)-\bar z_1(t)] dW(t)+  [z_2(t)-\bar z_2(t)] dW^{\bar u}(t)\\&
+\int_{E}[\Lambda (t,e)-\bar \Lambda(t,e)] \tilde \mu(dt,de)\\
x(0)-\bar x(0)=&0,\\
y(T)-\bar y(T)=&\phi(x(T))-\phi(\bar x(T)).
\end{aligned}
\right.
\end{eqnarray}

By \eqref{eq:18}, we know that
$(\bar p(\cdot),\bar q_1(\cdot),\bar q_2(\cdot), \bar q_3(\cdot,\cdot),\bar k(\cdot) )$ admits the
following   FBSDE

\begin{numcases}{}\label{eq:3.3}
\begin{split}
d\bar p\left(  t\right)
 &=-{\cal {H}}_{x}\left( t,\bar x(t),
 \bar y(t), \bar z_1(t),
 \bar z_2(t), \bar\Lambda(t,\cdot),  \bar u(t) \right)dt
 +\bar q_1
\left(  t\right)  dW\left(  t\right)  +{\bar q}_{2}\left( t\right) dW^{
\bar u}\left( t\right)
+\int_{E}\bar q_3 (t,e) \tilde \mu(dt,de),
\\
d\bar k\left(  t\right)  &=-{\cal {H}}_{y}\left(  t,\bar x(t),
 \bar y(t), \bar z_1(t),
 \bar z_2(t), \bar\Lambda(t,\cdot),  \bar u(t) \right)dt-
{\cal {H}}_{z_1}\left(  t,\bar x(t),
 \bar y(t), \bar z_1(t),
 \bar z_2(t), \bar\Lambda(t,\cdot),  \bar u(t) \right)  dW\left(  t\right)
  \\&\quad\quad-{\cal {H}}_{z_2}\left(  t,\bar x(t),
 \bar y(t), \bar z_1(t),
 \bar z_2(t), \bar\Lambda(t,\cdot),  \bar u(t) \right)dW^{
\bar u}\left( t\right)-
\int_{E}{\cal {H}}_{\Lambda}\left(  t,\bar x(t),
 \bar y(t), \bar z_1(t),
 \bar z_2(t), \bar\Lambda(t,\cdot),  \bar u(t) \right)\tilde \mu(dt,de),
\\ \bar p(T)&=\bar\Phi_x(T)-\bar \phi_x^\top(T)\bar k(T),
\\ \bar k(0)&=-\bar \gamma_y(0).
\end{split}
\end{numcases}
Moreover, we can easily obtain that $(\bar \Lambda(\cdot),
\bar R_1(\cdot),
\bar R_2(\cdot), \bar R_3(\cdot))$  admits the following BSDE

\begin{numcases}{}\label{eq:3.3}
\begin{split}
d\bar r(t)&=-[\bar l(t)+\bar R_2(t)\bar h(t)]dt+\bar R_1\left(
t\right) dW\left(  t\right) +{\bar R}_{2}\left( t\right) dY(t)+\int_E{\bar R}_{3}\left( t,e\right) \tilde \mu(dt,de),
\\ \bar r(T)&=\bar\Phi(T),
\end{split}
\end{numcases}
  In view of the definition of the cost function $J(u(\cdot)),$
it follows that
\begin{eqnarray}\label{eq:26}
\begin{split}
 J(u(\cdot))-J(\bar u(\cdot))=&
 \mathbb E^u\bigg[\int_0^Tl^u(t)dt+ \Phi^u(T)+\gamma^u(0)\bigg]-\mathbb E^{\bar u}\bigg[\int_0^T\bar l(t)dt+ \bar\Phi(T)+\bar\gamma(0)\bigg]
 \\=& \mathbb E\bigg[\int_0^T (\rho^u(t)l^u(t)-\bar\rho(t)\bar l(t))dt\bigg]+ \mathbb E[\rho^u(T)\Phi^u(T)-\bar \rho (t)\bar \Phi(T)]
 +\mathbb E[\gamma^u(0)-\bar \gamma(0)]
 \\=& \mathbb E^{\bar u}
 \bigg[\int_0^T [l^u(t)-\bar l(t)]dt
 \bigg]+ \mathbb E^{\bar u}[\Phi^u(T)-\bar \Phi(T)]+ \mathbb E\bigg[\int_0^T (\rho^u(t)-\bar\rho(t)) l^u(t)dt\bigg]
\\&+ \mathbb E[(\rho^u(T)-\bar \rho (T))\Phi^u(T)]
 +\mathbb E[\gamma^u(0)-\bar \gamma(0)]
\end{split}
\end{eqnarray}
In view of   the definition of $\cal H,$ it
follows that
\begin{eqnarray}\label{eq:27}
  \begin{split}
   \mathbb E^{\bar u}\bigg[\int_0^T (l^u(t)-\bar l(t))dt\bigg]=&{\mathbb E} ^{\bar u}
\bigg[\int_0^T \bigg ( { \cal H}(t, x^u(t),
y^u(t), z^u_1(t), z^u_2(t), \Lambda^u(t,\cdot),  u(t)) - {\cal H} (t, \bar x(t),
\bar y(t), \bar z_1(t),\bar z_2(t), \bar\Lambda(t,\cdot),  \bar u(t))\bigg)dt\bigg]
\\&- \mathbb E^{\bar u}\bigg[\int_0^T \bigg (\langle \bar p(t), b^u (t) - \bar b (t)\rangle + \langle \bar q_1(t), \sigma_1^u(t) - \bar \sigma_1 (t)\rangle
+\langle \bar  q_2(t), \sigma _2^u(t) - {\bar \sigma}_2 (t)
\rangle
+ \langle \bar k(t),  f^u(t) - \bar f (t)\rangle\\&+ \int_E\langle \bar q_3(t,e),  g^u(t,e) - \bar g (t,e)\rangle\nu(de)+\langle \bar R_2(t)-\bar \sigma_2^*(t)\bar p(t)-\bar z_2 ^*(t)\bar k(t),h^u(t)-\bar h(t)\rangle \bigg)dt\bigg]
  \end{split}
\end{eqnarray}

By using It\^{o} formula to $\la \bar p(t), x^u(t)-\bar x(t)\ra
+\la \bar k(t), y^u(t)-\bar y(t)\ra$ and taking expectation under $P^{\bar u}$, 
it follows that
\begin{eqnarray}
   && \mathbb E^{\bar u} \big[ \la \bar \Phi_x(T)-\bar\phi_x^*(T)\bar k(T), x^u(T)-\bar x(T)\ra\big]+\mathbb E^{\bar u} \big[ \la \bar k(T), \phi^u(T)-\bar \phi(T)\ra\big]\nonumber
\\=&&\mathbb E^{\bar u}\bigg[\int_0^T \bigg (\langle \bar p(t), b^u (t) - \bar b (t)\rangle + \langle \bar q_1(t), \sigma_1^u (t) - \bar \sigma_1 (t)\rangle
+ \langle \bar q_2(t), \sigma _2^u(t) - {\bar \sigma}_2 (t)\rangle+ \int_E\langle \bar q_3(t,e),  g^u(t,e) - \bar g (t,e)\rangle\nu(de)
\nonumber
\\&&+ \langle \bar k(t), f^u(t) - \bar f (t)\rangle
+\langle \bar p(t), \sigma_2^u(t)(\bar h(t)- h^u(t))\rangle +\langle \bar k(t), z_2^u(t)(\bar h(t)- h^u(t))\rangle \bigg)dt
\bigg]
\nonumber
\\&&-\mathbb E^{\bar u} \bigg[\int_0^T \la { \cal H}_x (t, \bar x(t),
\bar y(t), \bar z_1(t),\bar z_2(t), \bar\Lambda(t,\cdot),  \bar u(t)), x^u (t) - \bar
x (t)\ra dt\bigg]
\nonumber
\\&&-\mathbb E^{\bar u} \bigg[\int_0^T \la { \cal H}_y (t, \bar x(t),
\bar y(t), \bar z_1(t),\bar z_2(t), \bar\Lambda(t,\cdot),  \bar u(t)), y^u (t) - \bar
y (t)\ra dt\bigg]
\nonumber
\\&&-\mathbb E^{\bar u} \bigg[\int_0^T \la { \cal H}_{z_1}(t, \bar x(t),
\bar y(t), \bar z_1(t),\bar z_2(t), \bar\Lambda(t,\cdot),  \bar u(t)), z_1^u (t) - \bar
z_1 (t)\ra dt\bigg]
\nonumber
\\&&-\mathbb E^{\bar u} \bigg[\int_0^T \la { \cal H}_{z_2}(t, \bar x(t),
\bar y(t), \bar z_1(t),\bar z_2(t), \bar\Lambda(t,\cdot),  \bar u(t)), z_2^u (t) - \bar
z_2 (t)\ra dt\bigg]
\\&&-\mathbb E^{\bar u} \bigg[\int_0^T \int_{E}\big <{\cal H}_{\Lambda} (t, \bar x(t),
\bar y(t), \bar z_1(t),\bar z_2(t), \bar\Lambda(t,\cdot),  \bar u(t)),
 \Lambda^u_2 (t,e) - \bar
\Lambda_2 (t,e)\big>\nu(de)dt\bigg]
\nonumber
\\&&-{\mathbb E}
 \big [  \left <
y^u (0)
- \bar y (0), \bar \gamma_y (0)  \right
> \big ]
\bigg].
\end{eqnarray}
Therefore, we have
\begin{eqnarray} \label{eq:29}
   &&\mathbb E^{\bar u}\bigg[\int_0^T \bigg (\langle \bar p(t), b^u (t) - \bar b (t)\rangle + \langle \bar q_1(t), \sigma_1^u (t) - \bar \sigma_1 (t)\rangle
+ \langle \bar q_2(t), \sigma _2^u(t) - {\bar \sigma}_2 (t)\rangle+ \langle \bar k(t), f^u(t) - \bar f (t)\rangle \bigg)dt\bigg]\nonumber
\\=&& \mathbb E^{\bar u} \big[ \la \bar \Phi_x( T), x^u(T)-\bar x(T)\ra\big]+\mathbb E^{\bar u} \big[ \la \bar k(T), \phi^u(T)-\bar \phi(T)-\bar\phi_x(T)(x^u(T)-\bar x(T))\ra\big]\nonumber
\\&&+ \mathbb E\bigg[\int_0^T\bigg(\langle \bar p(t), \sigma_2^u(t)(h^u(t)-\bar h(t))\rangle +\langle \bar k(t), z_2^u(t)(h^u(t)-\bar h(t))\rangle \bigg)dt\bigg]\nonumber
\\&&+\mathbb E^{\bar u} \bigg[\int_0^T \la { \cal H}_x (t, \bar x(t),
\bar y(t), \bar z_1(t),\bar z_2(t), \bar\Lambda(t,\cdot),  \bar u(t)), x^u (t) - \bar
x (t)\ra dt\bigg]\nonumber
\\&&+\mathbb E^{\bar u} \bigg[\int_0^T \la { \cal H}_y (t, \bar x(t),
\bar y(t), \bar z_1(t),\bar z_2(t), \bar\Lambda(t,\cdot),  \bar u(t)), y^u (t) - \bar
y (t)\ra dt\bigg]\nonumber
\\&&+\mathbb E^{\bar u} \bigg[\int_0^T \la { \cal H}_{z_1}(t, \bar x(t),
\bar y(t), \bar z_1(t),\bar z_2(t), \bar\Lambda(t,\cdot),  \bar u(t)), z_1^u (t) - \bar
z_1 (t)\ra dt\bigg] \nonumber
\\&&+\mathbb E^{\bar u} \bigg[\int_0^T \la { \cal H}_{z_2}(t, \bar x(t),
\bar y(t), \bar z_1(t),\bar z_2(t), \bar\Lambda(t,\cdot),  \bar u(t)), z_2^u (t) - \bar
z_2 (t)\ra dt\bigg]\nonumber
\\&&+\mathbb E^{\bar u} \bigg[\int_0^T\int_E \la { \cal H}_{\Lambda}(t, \bar x(t),
\bar y(t), \bar z_1(t),\bar z_2(t), \bar\Lambda(t,\cdot),  \bar u(t)), \Lambda^u (t) - \bar
\Lambda (t)\ra \nu(de) dt\bigg]
\nonumber
\\&&+{\mathbb E}
 \big [  \left <
y^u (0)
- \bar y (0), \bar \gamma_y (0)  \right
> \big ]
\bigg].
\end{eqnarray}
Again using It\^{o} formula to
$(\rho^u(t)-\bar \rho(t)) \bar r(t), $
we get that
\begin{eqnarray}
  \begin{split}
    \mathbb E [ (\rho^u(T)-\bar\rho(T))\bar \Phi(T)]= &-\mathbb E\bigg[\int_0^T(\rho^u(t)-\bar\rho(t))(\bar l(t)+\bar R_2(t)\bar h(t))dt\bigg]
    +\mathbb E\bigg[\int_0^T\bar R_2(t)(\rho^u(t)h^u(t)-\bar\rho(t)\bar  h(t))dt\bigg].
  \end{split}
\end{eqnarray}
Therefore, we have
\begin{eqnarray}\label{eq:31}
  \begin{split}
    \mathbb E  [(\rho^u(T)-\bar\rho(T))\bar \Phi(T)]+\mathbb E\bigg[\int_0^T(\rho^u(t)
    -\bar\rho(t))\bar l(t)dt\bigg]=
    \mathbb E\bigg[\int_0^T\bar R_2(t)\rho^u(t)(h^u(t)-\bar  h(t))dt\bigg].
  \end{split}
\end{eqnarray}
By inserting \eqref{eq:29} into \eqref{eq:27}, it follows that
\begin{eqnarray}\label{eq:32}
  \begin{split}
   \mathbb E^{\bar u}
   \bigg[\int_0^T
    (l^u(t)-\bar l(t))dt\bigg]={\mathbb E} ^{\bar u}
\bigg[\int_0^T \bigg ( &{ \cal H}(t, x^u(t),
y^u(t), z^u_1(t), z^u_2(t), \Lambda^u(t,\cdot),  u(t)) - {\cal H} (t, \bar x(t),
\bar y(t), \bar z_1(t),\bar z_2(t), \bar\Lambda(t,\cdot),  \bar u(t))
\\& - \big < {\cal H}_x (t, \bar x(t),
\bar y(t), \bar z_1(t), \bar z_2(t), \bar\Lambda(t,\cdot),  \bar u(t)),x^u (t) - \bar
x (t)
\big>
\\&- \big <{ \cal H}_y (t, \bar x(t),
\bar y(t), \bar z_1(t), \bar z_2(t), \bar\Lambda(t,\cdot),  \bar u(t)),
y^u (t) - \bar
y (t)\big>
\\&- \big <{\cal H}_{z_1} (t, \bar x(t),
\bar y(t), \bar z_1(t),\bar z_2(t), \bar\Lambda(t,\cdot),  \bar u(t)),
 z^u_1 (t) - \bar
z_1 (t)\big>
\\&- \big <{\cal H}_{z_2} (t, \bar x(t),
\bar y(t), \bar z_1(t),\bar z_2(t), \bar\Lambda(t,\cdot),  \bar u(t)),
 z^u_2 (t) - \bar
z_2 (t)\big>
\\&- \int_{E}\big <{\cal H}_{\Lambda} (t, \bar x(t),
\bar y(t), \bar z_1(t),\bar z_2(t), \bar\Lambda(t,\cdot),  \bar u(t)),
 \Lambda^u_2 (t) - \bar
\Lambda_2 (t)\big>\nu(de)
\\&-
 \langle (\sigma_2^u(t)-\bar{\sigma}_2(t))
 (h^u(t)-\bar
 h(t)), \bar p(t)\rangle \bigg)dt\bigg] \\
& - {\mathbb E}^u \big [   \left <
x^u (T) - \bar x (T), \bar \Phi_x (T)  \right
> \big ]- {\mathbb E}
 \big [  \left <
y^u (0)
- \bar y (0), \bar \gamma_y (0)  \right
> \big ]
\\& - {\mathbb E}^u
\big [ \la \phi^u (T) - \bar \phi (T), k^u(T)\ra - \left <
\bar \phi_x^*(T)\bar k(T),
x^u (T) - \bar x (T) \right
> \big ].
  \end{split}
  \end{eqnarray}
Therefore, after we inserting\eqref{eq:31} and
 \eqref{eq:32} into \eqref{eq:26} ,  \eqref{eq:21} follows. The proof is complete.
 \end{proof}

For any given admissible controls $u (\cdot) \in {\cal A},$  because the control domain $U$ is convex, in $\cal A$ we can define
the following perturbed control process $u^\epsilon (\cdot)$:
\begin{eqnarray*}
u^\epsilon (\cdot) := \bar u (\cdot) + \epsilon ( u (\cdot) - \bar u (\cdot) ) , \quad  0 \leq \epsilon \leq 1.
\end{eqnarray*}
 Assume that $(\bar x (\cdot),
\bar y (\cdot),
\bar z_1 (\cdot),\bar z_2(\cdot),\bar \Lambda(\cdot),\bar \rho(\cdot))$ and
$(x^\epsilon (\cdot), y^\epsilon (\cdot), z^\epsilon_1 (\cdot), z^\eps_2(\cdot),
 \Lambda^\eps_(\cdot), \rho^\eps(\cdot))$ are the corresponding state processes corresponding to $\bar u (\cdot)$ and $u^\epsilon (\cdot)$, respectively. Assume that $(\bar p(\cdot),\bar q_1(\cdot), \bar { q}_2(\cdot), \bar q_3(\cdot),\bar k(\cdot),\bar \Lambda(\cdot),\bar R_1(\cdot),\bar R_2(\cdot ), \bar R_3(\cdot) )$ is the adjoint process
associated with the admissible pair $(\bar u (\cdot); \bar x (\cdot),
\bar y(\cdot), \bar z_1(\cdot),
 \bar z_2(\cdot),\bar \Lambda (\cdot),
 \bar \rho(\cdot))$.

\begin{lemma} \label{lem:3.5}
  Assume that Assumptions \ref{ass:1.1}
   and \ref{ass:1.2} hold. Then  
   it follows that
\begin{eqnarray*}
\begin{split}
&{\mathbb E} \bigg [ \sup_{t\in \cal T} | x^\epsilon (t) - \bar
x (t) |^4  \bigg ]
+{\mathbb E} \bigg [ \sup_{t\in \cal T} | y^\epsilon (t) - \bar
y (t) |^4  \bigg ]
+{\mathbb E} \bigg [
\bigg(\int_0^T| z^\epsilon_1 (t) - \bar
z _1(t) |^2 dt\bigg)^{2} \bigg ]
+{\mathbb E} \bigg [
\bigg(\int_0^T| z^\epsilon_2 (t) - \bar
z _2(t) |^2 dt\bigg)^{2} \bigg ]
\\&+{\mathbb E} \bigg [
\bigg(\int_0^T\int_{E}| z\Lambda^\epsilon (t,e) - \bar \Lambda(t,e) |^2 \nu(de)dt\bigg)^{2} \bigg ]= O (\epsilon^4) 
\end{split}
\end{eqnarray*}
and
\begin{eqnarray*}
&&{\mathbb E} \bigg [ \sup_{t\in \cal T} | \rho^\epsilon (t) - \bar
\rho (t) |^2  \bigg ]
= O (\epsilon^2) \ .
\end{eqnarray*}
\end{lemma}
\begin{proof}
  The proof can be obtained
  directly by Lemmas \ref{lem:3.3}
  and \ref{lem:3.4}.
\end{proof}

Now we are begin to apply  Lemma \ref{lem4} and Lemma \ref{lem:3.5} to get an variational formula for the cost functional
$J(u(\cdot))$  in view of  the Hamiltonian ${\cal H}$ and the adjoint process $(\bar p(\cdot),\bar q_1(\cdot), \bar { q}_2(\cdot), \bar q_3(\cdot),\bar k(\cdot),\bar \Lambda(\cdot),\bar R_1(\cdot),\bar R_2(\cdot ), \bar R_3(\cdot) )$.

\begin{theorem}\label{them:3.1}
 Assume that  Assumptions \ref{ass:1.1} and
 \ref{ass:1.2} holds.
Then for any admissible control $u (\cdot) \in {\cal A}$, an variational formula for the cost functional
$J(u(\cdot))$ is given by
\begin{eqnarray}\label{eq:4.4}
&& \frac{d}{d\epsilon} J ( \bar u (\cdot) + \epsilon ( u (\cdot) - \bar u (\cdot) ) ) |_{\epsilon=0} \nonumber \\
&& := \lim_{\epsilon \rightarrow 0^+}
\frac{ J ( \bar u (\cdot) + \epsilon ( u (\cdot) - \bar u (\cdot) ) )
- J( \bar u (\cdot) ) }{\epsilon} \nonumber \\
&& = {\mathbb E}^{\bar u} \bigg [ \int_0^T \left < {\cal H}_u (t, \bar x(t),
\bar y(t), \bar z_1(t),\bar z_2(t), \bar\Lambda(t,\cdot),  \bar u(t)),
u (t) -
 \bar u (t) \right > d t \bigg ] .
\end{eqnarray}
\end{theorem}

\begin{proof}
To simplify  our notations, we define
\begin{eqnarray}
\alpha^\epsilon &:=& {\mathbb E} ^{\bar u}
\bigg[\int_0^T \bigg [ { \cal H}(t, x^{u^\eps}(t),
y^{u^\eps}(t), z_1^{u^\eps}(t),z_2^{u^\eps}(t), \Lambda_2^{u^\eps}(t,\cdot),
u(t)) - {\cal H} (t, \bar x(t),
\bar y(t), \bar z_1(t),\bar z_2(t), \bar\Lambda(t,\cdot),  \bar u(t)) \nonumber \\&& - \big < {\cal H}_x (t, \bar x(t),
\bar y(t), \bar z_1(t),\bar z_2(t), \bar\Lambda(t,\cdot),  \bar u(t)),x^{u^\eps} (t) - \bar
x (t)
\big>- \big <{ \cal H}_y (t, \bar x(t),
\bar y(t), \bar z_1(t),
\bar z_2(t), \bar\Lambda(t,\cdot),  \bar u(t)),
y^{u^\eps} (t) - \bar
y (t)\big>\nonumber
\\&&- \big <{\cal H}_{z_1} (t, \bar x(t),
\bar y(t), \bar z_1(t),
\bar z_2(t), \bar\Lambda(t,\cdot),  \bar u(t)),
 z_1^{u^\eps} (t) - \bar
z_1 (t)\big>- \big <{\cal H}_{z_2} (t, \bar x(t),
\bar y(t), \bar z_1(t),\bar z_2(t), \bar\Lambda(t,\cdot),  \bar u(t)),
 z_2^{u^\eps} (t) - \bar
z_2(t)\big>\nonumber
\\&& - \big <{\cal H}_{u} (t, \bar x(t),
\bar y(t), \bar z_1(t),\bar z_2(t), \bar\Lambda(t,\cdot),  \bar u(t)),
 u^\eps (t) - \bar
u(t)\big>\nonumber
\\&&-\int_E \big <{\cal H}_{\Lambda} (t, \bar x(t),
\bar y(t), \bar z_1(t),\bar z_2(t), \bar\Lambda(t,\cdot),\bar u(t)),
 \Lambda^\eps (t,e) - \bar
\Lambda(t,e)\big>\nu(de)
 \\&&-\langle (\sigma_2^{u^\eps}(t)-\bar{\sigma}_2(t))
 (h^{u^\eps}(t)-\bar
 h(t)), \bar p(t)\rangle \nonumber
 \\&& -
 \langle (z_2^{u^\eps}(t)-\bar{z}_2(t))
 (h^{u^\eps}(t)-\bar
 h(t)), \bar k(t)\rangle dt\bigg]
\nonumber \\
&& + {\mathbb E}^{\bar u} \big [ \Phi^{u^\eps}(T)
 - \bar \Phi (T) - \left <
x^{u^\eps} (T) - \bar x (T), \bar \Phi_x (T)  \right
> \big ]
\nonumber \\
&& - {\mathbb E}^{\bar u}
\big [ \la \phi^{u^\eps} (T) - \bar \phi (T), \bar k(T)\ra - \left <
\bar \phi_x^*(T) \bar k(T),
x^{u^\eps} (T) - \bar x (T) \right
> \big ]
\nonumber \\
&& + {\mathbb E}
 \big [ \gamma^{u^\eps}(0) - \bar \gamma (0) - \left <
y^{u^\eps} (0)
- \bar y (0), \bar \gamma_y (0)  \right
> \big ]
\nonumber\nonumber
\\&&+\mathbb E\bigg[\int_0^T \bar R_2(t)(\rho^{u^\eps}(t)-\rho^{\bar u}(t))(h^{u^\eps}(t)-\bar
 h(t))dt\bigg] \nonumber
\\&&+\mathbb E\bigg[\int_0^T(l^{u^\eps}(t)-\bar l(t))( \rho^{u^\eps}(t)-\bar
\rho(t))dt\bigg]\nonumber
\\&&
+\mathbb E \bigg[ ( \rho^{u^\eps} (T) - \bar \rho(T)) (\Phi^{u^\eps} (T) -
\bar \Phi (T))\bigg].
\end{eqnarray}
 It follows from Lemma \ref{lem4} that
\begin{eqnarray}\label{eq:4.13}
J ( u^\epsilon (\cdot) ) - J ( \bar u (\cdot) )
= \beta^\epsilon + \epsilon {\mathbb E}^
{\bar u} \bigg [ \int_0^T \left < {\cal H}_u(t, \bar x(t),
\bar y(t), \bar z_1(t),\bar z_2(t), \bar\Lambda(t,\cdot),  \bar u(t)), u (t) - \bar u (t) \right > d t \bigg ] .
\end{eqnarray}
In view of  Assumptions \ref{ass:1.1} and
\ref{ass:1.2}, by using the Taylor Expansions, Lemma \ref{lem:3.5}, and the dominated
convergence theorem, it follows that
\begin{eqnarray}\label{eq:4.121}
\beta^\epsilon = o(\epsilon) .
\end{eqnarray}
Inserting \eqref{eq:4.121} into \eqref{eq:4.13}, we obtain that
\begin{eqnarray*}
\lim_{\epsilon \rightarrow 0^+} \frac{J (u^\epsilon (\cdot) ) - J (\bar u (\cdot))}{\epsilon}
= {\mathbb E}^{\bar u} \bigg [ \int_0^T \left < {\cal H}_u (t, \bar x(t-),
\bar y(t-), \bar z_1(t),\bar z_2(t), \bar\Lambda(t,\cdot),  \bar u(t)),
u (t) - \bar u (t) \right > d t \bigg ] .
\end{eqnarray*}
This completes the proof.
\end{proof}

\section{Necessary and Sufficient  Stochastic Maximum Principles }
  The purpose of this section is 
  to derive the necessary and sufficient maximum principles
   for Problem \ref{pro:1.1} or \ref{pro:1.2} in a unified way 
   by the results established in 
   Section 2.

\begin{theorem}[{\bf Necessary Stochastic Maximum principle}]
 Let Assumptions \ref{ass:1.1}
  and \ref{ass:1.2} hold. Suppose $( \bar u (\cdot); \bar x (\cdot), \bar y(\cdot), \bar z_1(\cdot),
  \\\bar z_2(\cdot),\bar\rho(\cdot) )$ be
an optimal pair of  Problem \ref{pro:1.2}
with the corresponding 
adjoint process $(\bar p(\cdot),\bar q_1(\cdot), \bar { q}_2(\cdot),\bar { q}_3(\cdot,\cdot),
\bar k(\cdot), \bar r(\cdot),\bar R_1(\cdot),\bar R_2(\cdot ),\\\bar { R}_3(\cdot,\cdot) )$. Using the short-hand notation, then the
following local maximum condition 
holds:
\begin{eqnarray}\label{eq:4.15}
\left <  {\mathbb E}^{\bar u} [ {\cal H}_{u} (t, \bar x(t-),
\bar y(t-), \bar z_1(t),\bar z_2(t), \bar\Lambda(t,\cdot),  \bar u(t))  | {\mathscr F}_t ^Y ],
v - \bar u (t) \right > \geq 0 , \quad \forall v \in U ,\ a.e.\ a.s..
\end{eqnarray}
\end{theorem}

\begin{proof}
 For any admissible control 
 $u(\cdot)\in \cal A,$ in view of $\mathscr P^Y$-measurable
 property of  admissible controls,  Theorem \ref{them:3.1}, the optimality of $\bar u(\cdot)$ and the property of conditional
expectation, we obtain
\begin{eqnarray*}
&& {\mathbb E} \bigg [ \int_0^T
 \langle {\mathbb E}
  [{\bar \rho(t)}{ \bar {\cal H}}_u (t, \bar x(t-),
\bar y(t), \bar z_1(t),\bar z_2(t), \bar\Lambda(t,\cdot),  \bar u(t)) | {\mathscr F}_t ^Y] ,
u (t) - \bar u (t) \rangle  d t \bigg ] \\
&& = {\mathbb E} \bigg[\int_0^T \langle
{\bar \rho(t)}{ \bar  {\cal H}}_u (t, \bar x(t-),
\bar y(t-), \bar z_1(t),\bar z_2(t), \bar\Lambda(t,\cdot),  \bar u(t))
, u (t) - \bar u (t) \rangle \d t \bigg ]
\\
&& = {\mathbb E}^{\bar u} \bigg[\int_0^T \langle{ \bar {\cal H}}_u (t, \bar x(t-),
\bar y(t-), \bar z_1(t),\bar z_2(t), \bar\Lambda(t,\cdot),  \bar u(t))
, u (t) - \bar u (t) \rangle \d t \bigg ] \\
&& = \lim_{\epsilon \rightarrow 0^+} \frac{J( \bar u (\cdot) + \epsilon (
u (\cdot) - \bar u (\cdot) ) ) - J ( \bar u (\cdot) )}{\epsilon} \geq 0 .
\end{eqnarray*}
This implies that
\begin{eqnarray}
  \langle {\mathbb E}
  [{\bar \rho(t)}{ \bar {\cal H}}_u ( t, \bar x(t-),
\bar y(t-), \bar z_1(t),\bar z_2(t), \bar\Lambda(t,\cdot),  \bar u(t)) | {\mathscr F}_t ^Y] ,
v - \bar u (t) \rangle \geq 0, \quad \forall v \in U ,\ a.e.\ a.s.,
\end{eqnarray}
since $u(\cdot)\in \cal A$ is 
arbitrary admissible control. 
Therefore, since
$\bar \rho(t)>  0, $
by using Baye's rule for conditional expectations, we get that 
 \begin{eqnarray}
   \begin{split}
    & \left <  {\mathbb E}^{\bar u} [ {\cal H}_{u} (t, \bar x(t-),
\bar y(t-), \bar z_1(t),\bar z_2(t), \bar\Lambda(t,\cdot),  \bar u(t))  | {\mathscr F}_t ^Y ],
v - \bar u (t) \right >
\\&=
\frac{1}{\mathbb E[\bar\rho (t)|{\mathscr F}_t ^Y]}
\big\langle {\mathbb E}
  [{\bar \rho(t)}{ \bar {\cal H}}_u (t, \bar x(t-),
\bar y(t-), \bar z_1(t),\bar z_2(t), \bar\Lambda(t,\cdot),  \bar u(t)) | {\mathscr F}_t ^Y] ,
v - \bar u (t)\big \rangle
\\& \geq 0.
   \end{split}
 \end{eqnarray}
  The proof is complete.

\end{proof}

 In the following, we discuss  the sufficient maximum principle for  an optimal control of
Problem \ref{pro:1.2} in a special case when
the observation process is not
affected by the control process.
 More precisely,  in our observation equation \eqref{eq:2}, we assume that
$$h(t,x,u)=h(t)$$ is an
$\mathscr P^Y-$measurable bounded
process.
  Define a new probability measure $\mathbb Q$ on $(\Omega, \mathscr F)$ by
\begin{eqnarray}
  d\mathbb Q=\rho(1)d\mathbb P,
\end{eqnarray}
where
\begin{eqnarray} \label{eq:433}
  \left\{
\begin{aligned}
  d \rho(t)=&  \rho(t) h(s)dY(s)\\
  \rho(0)=&1.
\end{aligned}
\right.
\end{eqnarray}

\begin{theorem}{\bf [Sufficient Maximum Principle] } \label{thm:4.1}
 Assume that Assumptions \ref{ass:1.1}
  and \ref{ass:1.2} hold. Suppose that $(\bar u (\cdot);
  \bar x (\cdot),
  \bar y(\cdot), \bar z_1(\cdot),\\
  \bar z_2 (\cdot), \bar \Lambda(\cdot))$ be an admissible pair with $\phi(x)=\phi x,$
  where $\phi$ is $\mathscr F_T-$measurable bounded  random variable.
Moreover, we assume that  
\begin{enumerate}
\item[(i)] the Hamiltonian ${\cal H}$ is convex in $(x,y,z_1,z_2, \Lambda(\cdot),u) \in \mathbb {R}^n\times \mathbb  R^m\times \mathbb R^m\times \mathbb R^m \times {M}^{\nu,2}( E; \mathbb R^m)\times
U$,  and $\Phi$ and $\gamma$ are convex in $x$ and $y,$ respectively,
\item[(ii)]
\begin{eqnarray*}\label{eq:5.119}
&& \mathbb E\bigg[{\cal H} ( t,
\bar x (t-),
\bar y (t-), \bar z_1 (t),
\bar z_2 (t),\bar\Lambda(t,\cdot),\bar u(t)) |\mathscr F^Y_t\bigg]\nonumber \\
&& = \min_{u \in U }
 \mathbb E\bigg[{\cal H} ( t,
\bar x (t-),
\bar y (t-), \bar z_1 (t),\bar z_2 (t),\bar \Lambda(t,\cdot), u)
 |\mathscr F^Y_t\bigg], \quad \mbox {a.e.\ a.s.} ,
\end{eqnarray*}
\end{enumerate}
then $(\bar u (\cdot),  \bar x (\cdot),
\bar y(\cdot), \bar z_1(\cdot), \bar z_2 (\cdot),\bar \Lambda(\cdot,\cdot))$ is an optimal pair of Problem \ref{pro:1.2}.
\end{theorem}

\begin{proof}
Given  an arbitrary admissible pair$({u}(\cdot);x^u(\cdot),
 y^u(\cdot),
 z^u_1(\cdot),  z^u_2(\cdot),
\Lambda^u(\cdot,\cdot)),$  from Lemma \ref{lem4},
we can get the following formula for the difference $J ( u (\cdot) ) - J ( \bar u (\cdot) )$:
\begin{eqnarray}\label{eq:40}
&&J (u (\cdot)) - J (\bar u(\cdot))\nonumber
\\ &=& {\mathbb E} ^Q
\bigg[\int_0^T \bigg [ { \cal H}(t,x^u(\cdot),
 y^u(\cdot),
 z^u_1(\cdot),  z^u_2(\cdot),
\Lambda^u(\cdot,\cdot), u(t)) - {\cal H} (t, \bar x(t),
\bar y(t), \bar z_1(t),\bar z_2(t), \bar\Lambda(t,\cdot),  \bar u(t)) \nonumber \\&& - \big < {\cal H}_x (t, \bar x(t),
\bar y(t), \bar z_1(t),\bar z_2(t), \bar\Lambda(t,\cdot),  \bar u(t)),x^u (t) - \bar
x (t)
\big>- \big <{ \cal H}_y (t, \bar x(t),
\bar y(t), \bar z_1(t),\bar z_2(t), \bar\Lambda(t,\cdot),  \bar u(t)),
y^u (t) - \bar
y (t)\big>\nonumber
\\&&- \big <{\cal H}_{z_1} (t, \bar x(t),
\bar y(t), \bar z_1(t),\bar z_2(t), \bar\Lambda(t,\cdot),  \bar u(t)),
 z^u_1 (t) - \bar
z _1(t)\big>- \big <{\cal H}_{z_2} (t, \bar x(t),
\bar y(t), \bar z_1(t),\bar z_2(t), \bar\Lambda(t,\cdot),  \bar u(t)),
 z^u_2 (t) - \bar
z _2(t)\big>\nonumber
\\&&
 -\int_{E}\big <{\cal H}_{\Lambda} (t, \bar x(t),
\bar y(t), \bar z_1(t),\bar z_2(t), \bar\Lambda(t,\cdot),  \bar u(t)),
\Lambda (t,e) - \bar\Lambda(t,e)\big>\nu(de)\bigg]dt\bigg]
\nonumber \\
&& + {\mathbb E}^{Q} \big [ \Phi(T)
 - \bar \Phi^u (T) - \left <
x^u (T) - \bar x (T), \bar \Phi_x (T)  \right
> \big ]
\nonumber \\
&& + {\mathbb E}
 \big [ \gamma^u(0) - \bar \gamma (0) - \left <
y^u (0)
- \bar y (0), \bar \gamma_y (0)  \right
> \big ].
\end{eqnarray}

By the convexity of ${\cal H}$, $\Phi$
and $\gamma$ (i.e. Conditions (i)), we have
\begin{eqnarray}\label{eq:41}
 &&{ \cal H}(t, x^u(t),
y^u(t), z^u_1(t), z^u_2(t), \Lambda^u(t,\cdot),  u(t)) - {\cal H} (t, \bar x(t),
\bar y(t), \bar z_1(t),\bar z_2(t), \bar\Lambda(t,\cdot),  \bar u(t))\nonumber
\\ &\geq&  \big < {\cal H}_x (t, \bar x(t),
\bar y(t), \bar z_1(t), \bar z_2(t), \bar\Lambda(t,\cdot),  \bar u(t)),x^u (t) - \bar
x (t)
\big>\nonumber
+ \big <{ \cal H}_y (t, \bar x(t),
\bar y(t), \bar z_1(t), \bar z_2(t), \bar\Lambda(t,\cdot),  \bar u(t)),
y^u (t) - \bar
y (t)\big>\nonumber
\\&&+ \big <{\cal H}_{z_1} (t, \bar x(t),
\bar y(t), \bar z_1(t),\bar z_2(t), \bar\Lambda(t,\cdot),  \bar u(t)),
 z^u_1 (t) - \bar
z_1 (t)\big>\nonumber
+ \big <{\cal H}_{z_2} (t, \bar x(t),
\bar y(t), \bar z_1(t),\bar z_2(t), \bar\Lambda(t,\cdot),  \bar u(t)),
 z^u_2 (t) - \bar
z_2 (t)\big>
\\&&+ \int_{E}\big <{\cal H}_{\Lambda} (t, \bar x(t),
\bar y(t), \bar z_1(t),\bar z_2(t), \bar\Lambda(t,\cdot),  \bar u(t)),
\Lambda (t,e) - \bar\Lambda(t,e)\big>\nu(de)\nonumber
\\&&+ \big <{\cal H}_{u} (t, \bar x(t),
\bar y(t), \bar z_1(t),\bar z_2(t), \bar\Lambda(t,\cdot),  \bar u(t)),u (t) - \bar
u(t)\big>,
\end{eqnarray}

\begin{eqnarray}\label{eq:42}
 \Phi^u(T) - \bar \Phi (T) \geq \left < X^u (T) - \bar X (T), \Phi_x (T) \right >
\end{eqnarray}
and
\begin{eqnarray}\label{eq:43}
 \gamma^u(0) - \bar \gamma (0) \geq \left < y^u (0) - \bar y (0), \gamma_y (0) \right >.
\end{eqnarray}

Furthermore, in view of   the convex optimization principle (see Proposition 2.21 of \cite{ET1976}) and  the optimality condition (ii), we have
\begin{eqnarray}\label{eq:5.5}
\left < u (t) - \bar u (t), \mathbb E\bigg[{\cal H}_{u} ( t,
\bar x (t),
\bar y (t), \bar z_1 (t),
\bar z_2(t),\bar\Lambda(t,\cdot),\bar u(t)) |\mathscr F^Y_t\bigg] \right >
\geq 0 .
\end{eqnarray}
Since $h(t)$ is an
$\mathscr P^Y-$measurable bounded
process, from \eqref{eq:433}, we know 
that  $\rho(\cdot)$ is a positive 
$\mathscr F_t-$ adopted 
process which implies that 

\begin{eqnarray}\label{eq:45}
\begin{split}
&\mathbb E^Q\bigg[\int_0^T\left < u (t) - \bar u (t), {\cal H}_{u} ( t,
\bar x (t),
\bar y (t), \bar z_1(t),\bar z_2(\cdot),
\bar\Lambda(t,\cdot),\bar u(t)) \right >
\bigg]\\&
=\mathbb E\bigg[\int_0^T\rho(t)\left < u (t) - \bar u (t), {\cal H}_{u} ( t,
\bar x (t),
\bar y (t), \bar z_1(t),\bar z_2(\cdot),
\bar\Lambda(t,\cdot),\bar u(t)) \right >
\bigg]
\\&
=\mathbb E\bigg[\int_0^T\rho(t)\left < u (t) - \bar u (t), \mathbb E[{\cal H}_{u} ( t,
\bar x (t),
\bar y (t), \bar z_1(t),\bar z_2(\cdot),
\bar\Lambda(t,\cdot),\bar u(t))|\mathscr F_{t}^Y] \right >
\bigg]
\\&\geq 0.
\end{split}
\end{eqnarray}
Inserting \eqref{eq:41},\eqref{eq:42},\eqref{eq:43} and  \eqref{eq:45} into \eqref{eq:40},
we obtain
\begin{eqnarray}
J (u (\cdot)) - J (\bar u (\cdot)) \geq 0.
\end{eqnarray}
 Because of the arbitrariness of $u (\cdot)$, we get 
  the optimality of $\bar u (\cdot)$
and thus $(\bar u (\cdot), \bar x (\cdot), \bar y(\cdot), \bar z_1(\cdot), z_2(\cdot), \bar \Lambda(\cdot,\cdot))$ is an optimal pair. The proof is completed.
\end{proof}
\section{Conclusion}
In this paper, partial observable optimal control problem for forward-backward 
stochastic system 
driven by Brownian motion and Poisson 
random Martingale measure has been discussed. Necessary and sufficient conditions, in the form of Pontryagin maximum principle, for partial information optimal control are obtained in a unified way. In the future, we will study partial observable differential games problem for optimal control problem  forward-backward
stochastic system of jump diffusion.

\bibliographystyle{model1a-num-names}

\begin{thebibliography}{00}




\bibitem{CT2010}
Chen, S., \& Tang, S. (2010). Semi-linear Backward Stochastic Integral Partial Differential Equations driven by a Brownian motion and a Poisson point process. arXiv preprint arXiv:1007.3201.


\bibitem{ET1976} Ekeland, I., T\'{e}mam, R. (1976). Convex Analysis and Variational Problems. North-Holland, Amsterdam.


\bibitem{FT2015}Fujii, M., \& Takahashi, A. (2015). Asymptotic Expansion for Forward-Backward SDEs with Jumps. arXiv preprint arXiv:1510.03220.

\bibitem{MST}
Meng, Q., Shi, Q., \& Tang, M. (2017). A Revisit to Optimal Control of Forward-Backward Stochastic Differential System with Observation Noise. arXiv preprint arXiv:1708.03008.


\bibitem{ML2016}
Ma, H., Liu, B. (2016). Linear-Quadratic Optimal Control Problem for Partially Observed Forward-Backward Stochastic Differential Equations of Mean-Field Type. Asian Journal of Control, 18(6), 2146-2157.


\bibitem{Shi2010}Shi, J. (2010, May). The maximum principle for partially observed optimal control of fully coupled forward-backward stochastic systems with state constraints. In Control and Decision Conference (CCDC), 2010 Chinese (pp. 572-578). IEEE.



\bibitem{SZ2013}Shi, Y.,\& Zhu, Q. (2013). Partially observed optimal controls of forward-backward doubly stochastic systems. ESAIM: Control, Optimisation and Calculus of Variations, 19(3), 828-843.



\bibitem{WW2009}
Wang, G., \& Wu, Z. (2009). The maximum principles for stochastic recursive optimal control problems under partial information. IEEE Transactions on Automatic control, 54(6), 1230-1242.







\bibitem{WWX2015}Wang, G., Wu, Z., \& Xiong, J. (2015). A linear-quadratic optimal control problem of forward-backward stochastic differential equations with partial information. IEEE Transactions on Automatic Control, 60(11), 2904-2916.










   

\bibitem{WWX}Wang, G., Wu, Z., \& Xiong, J. (2013). Maximum principles for forward-backward stochastic control systems with correlated state and observation noises.
      {\it SIAM Journal on Control and Optimization,} 51(1), 491-524.

\bibitem{WWZ}Wang, G., Wu, Z., \& Zhang, C. (2016, July). A partially observed optimal control problem for mean-field type forward-backward stochastic system. In Control Conference (CCC), 2016 35th Chinese (pp. 1781-1786). IEEE.


\bibitem{xiao2011}
Xiao, H. (2011). The maximum principle for partially observed optimal control of forward-backward stochastic systems with random jumps. Journal of Systems Science and Complexity, 24(6), 1083-1099.


\bibitem{XZZ2016} Xiong, J., Zhang, S.,
 \& Zhuang, Y. (2016). A partially observed non-zero sum differential game of forward-backward stochastic differential equations and its application in finance. arXiv preprint arXiv:1601.00538.
 
 
 
 
 
 \bibitem{ZRW}Zhou, Q., Ren, Y., \& Wu, W. On optimal mean-field control problem of mean-field forward-backward stochastic system with jumps under partial information. Journal of Systems Science and Complexity, 1-29.
     
     
     
     

\end{thebibliography}

\end{document}